\documentclass[psamsfonts]{amsart}

\usepackage{amssymb,amsfonts, mathrsfs}
\usepackage[all,arc]{xy}
\usepackage{enumerate}
\usepackage{mathrsfs}
\usepackage{changepage}
\usepackage{graphicx}
\usepackage{calc}
\usepackage{tikz}
\usepackage{float}
\usetikzlibrary{cd}

\newtheorem{thm}{Theorem}[section]

\newtheorem{lem}[thm]{Lemma}

\theoremstyle{definition}
\newtheorem{defn}[thm]{Definition}

\theoremstyle{remark}
\newtheorem{rem}[thm]{Remark}

\makeatletter
\let\c@equation\c@thm
\makeatother
\numberwithin{equation}{section}

\newcommand*{\img}[1]{%
    \raisebox{-.3\baselineskip}{%
        \includegraphics[
        height=\baselineskip,
        width=\baselineskip,
        keepaspectratio,
        ]{#1}%
    }%
}

\newcommand*\circled[1]{\tikz[baseline=(char.base)]{
    \node[shape=circle,draw,inner sep=2pt] (char) {#1};}}
    
\tikzset{commutative diagrams/.cd,
mysymbol/.style={start anchor=center,end anchor=center,draw=none}
}

\newcommand{\C}{\mathbb{C}}
\newcommand{\R}{\mathbb{R}}
\renewcommand{\O}{\mathcal{O}}
\renewcommand{\P}{\mathcal{P}}
\newcommand{\Pd}{\mathscr{P}_d}
\newcommand{\I}{\mathcal{I}}
\newcommand{\U}{\mathcal{U}}
\newcommand{\Bd}{\operatorname{Bd}}

\newcommand{\ZZ}{\mathbb{Z} / 2\mathbb{Z}}
\newcommand{\ddx}{\frac{\mathrm{d}}{\mathrm{d} \, x}}

\newcommand{\CmD}{\C^d \setminus \Delta}
\newcommand{\PmS}{\Pd \setminus \Sigma}

\newcommand{\Id}{\mathrm{Id}}

\newcommand{\inc}{\subset}
\newcommand{\tto}{\rightarrow}

\renewcommand{\Re}{\mathrm{Re}}

\title{The Cup Length in Cohomology as a Bound on topological complexity}

\author{Parth Sarin}

\date{Summer 2017}

\begin{document}

\begin{abstract}

Polynomial solving algorithms are essential to applied mathematics and the sciences. As such, reduction of their complexity has become an incredibly important field of topological research. We present a topological approach to constructing a lower bound for the complexity of a polynomial-solving algorithm, and give a concrete algorithm to do this in the case that $\deg(f) = 2,3,4$.

\end{abstract}

\maketitle

\tableofcontents

\section{Defining the problem}

When computers execute programs to solve polynomials, it is inevitable that such programs include ``if'' statements. These decisions create branching in the algorithm. We would like to bound the amount of branching in algorithms. Formally, the problem at hand is:

\begin{adjustwidth*}{2.5em}{0pt}
	Poly($d$): Given a complex monic polynomial of degree $d$, find the roots of $f$ within $\epsilon$.
\end{adjustwidth*}

Given our assumptions, we can model any polynomial-solving algorithm with a tree. (For the present paper, we exclude any programs with loops in them.) More formally, we adopt the convention of Manna's Flowchart Programs \cite{Manna}.

\begin{defn}
	An \textbf{algorithm} is a rooted tree with a root at the top representing the input and leaves at the bottom representing the output. There are two types of internal nodes:
	
	\textit{Computation nodes}, \img{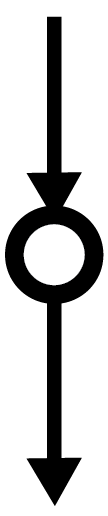}, which output their input modified by a rational function.
	
	\textit{Decision nodes}, \img{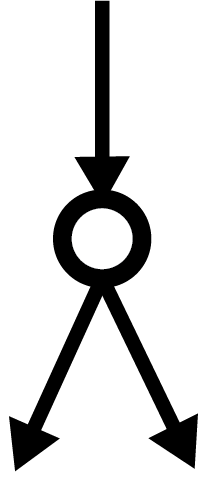}, at which the program either goes left or right depending on whether some inequality is true or false.
\end{defn}

We call these algorithms a ``computation tree.''

\begin{defn}
	The \textbf{topological complexity} of an algorithm is the number of decision nodes (intuitively, topologically speaking, computation nodes don't affect the structure of the tree.)
\end{defn}

The main theorem of Smale \cite{Smale}, which we will prove in this paper is:

\begin{thm} [Smale] \label{main-thm}
	For all $0 < \epsilon < \epsilon(d)$, the topological complexity of Poly($d$) is (strictly) greater than $(\log_2 (d))^{2/3} - 1$.
\end{thm}

\section{Algorithmic processes and nodes}

Essentially, an algorithm takes input of a certain type (from a specific space, in our case a subset of the complex monic polynomials), then does computation, and stores intermediary variables in ``memory'' (also represented as a specific space), and then returns output (another space, in our case, a $d$-tuple of complex numbers).

It is important to formalize this notion. The input domain is denoted $\I$, the output domain is denoted $\O$, and the program (computation) domain is denoted $\P$.

There are four types of internal nodes in an algorithm: start, compute, decision, and end.

\vspace{3mm}

\begin{adjustwidth*}{0em}{2.5em}
	Start nodes are defined by a rational map, $f : \I \tto \P$ and perform the first computation on the input, $x$:
	
	\begin{center}
	\includegraphics[scale=0.3]{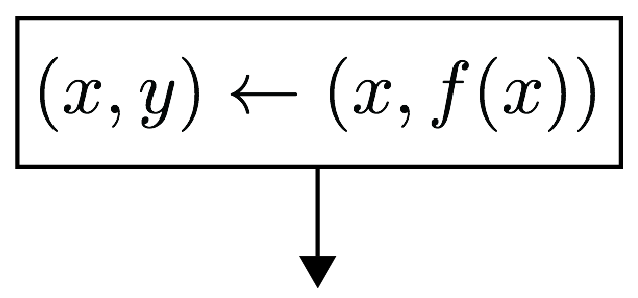}
	\end{center}
	
	\noindent
	Computation nodes are defined by a rational map, $g: \I \times \P \tto \P$:
	
	\begin{center}
	\includegraphics[scale=0.3]{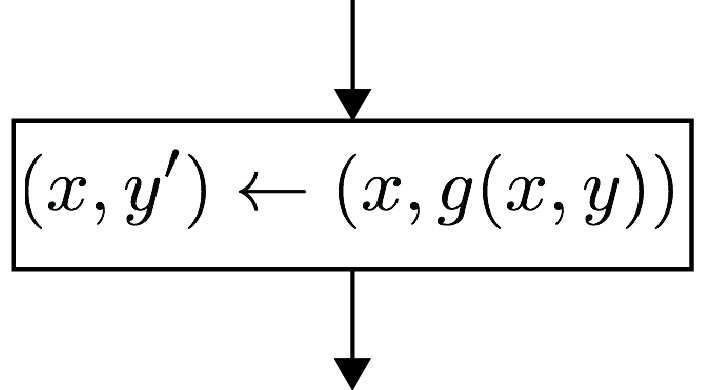}
	\end{center}
	
	\noindent
	Each branch (decision node) is a rational map, $h : \I \times \P \tto \P$, and upon encountering the branch, the computer makes a decision based on an inequality (could be $<$ or $\leq$):
	
	\begin{center}
		\includegraphics[scale=0.3]{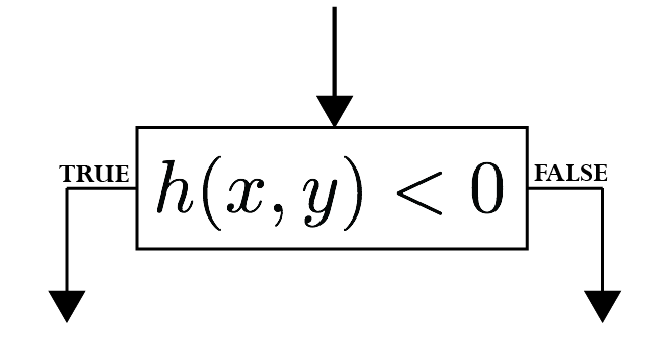}
	\end{center}
	
	\noindent
	Finally, end nodes are defined by a rational map, $l : \I \times \P \tto \O$, and perform the final computation:
	
	\begin{center}
		\includegraphics[scale=0.3]{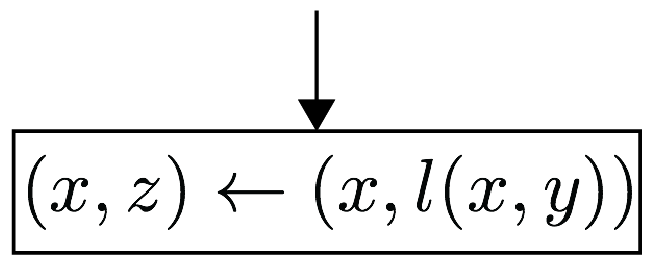}
	\end{center}
\end{adjustwidth*}

\vspace{3mm}

Each input, $x \in \I$ defines a specific path down the tree.

\section{From computation theory to algebraic topology}

This section is based on a spectacular result by Smale \cite{Smale}, which establishes a connection between computation theory and algorithmic complexity and the covering space theory, so the rest of the proof can proceed by algebraic topology.

The program will not use all the inputs in $\I$ (in our case, we want to only consider polynomials with distinct roots). The set of \textit{usable inputs} is a semialgebraic set $Y \inc \I$. Similarly, the set of \textit{usable outputs} is the set of potential outputs, when the program is given elements of $Y$. Formally, the algorithm returns an element of $Y \times \O$, so the usable outputs is the set $X \inc Y \times \O$ such that $X \tto Y$ is the restriction of the projection $Y \times \O \tto Y$ where the projection is required to be surjective. Additionally, we require that if $x \in Y$, then division by zero is never encountered in the algorithm.

Let $f : X \tto Y$ be continuous.

Now, we specify $\I$, $\P$, $\O$, $X$, and $Y$ for Poly($d$).

The input space, $\I$ is $\Pd \inc \C[t]$, the space of complex monic polynomials of degree $d$. $p \in \Pd$ can be written as $p = t^d + a_{d-1} t^{d-1} + \cdots + a_1 t + a_0$, so we can identify $\Pd$ with $\C^d$.

The program space is just $\C^n$ with $n$ sufficiently large.

The program is tasked with approximating the roots of a polynomial, so, naturally, the output space $\O$ is $\C^d$, and is populated by $d$-tuples of the space of polynomials.

Since restriction of the input space does not increase the complexity, we consider $Y \inc \I$ to be the collection of polynomials with distinct roots, $X \inc Y \times \O$ to be the corresponding output (roots of the polynomials of $Y$ to within $\epsilon$), and $f : X \tto Y$ to be the restriction of the projection map $Y \times \O \tto Y$.

Let the function $\pi : \C^d \tto \Pd$ be the function which maps roots $(\gamma_1, \cdots, \gamma_d)$ to the monic polynomial which vanishes on the $\gamma_i$. That is,
\[ \pi(\gamma_1, \cdots, \gamma_d) = (t-\gamma_1) \cdots (t-\gamma_d) \]
And the coefficients of $\pi$ in the natural basis of $\Pd$ are given by the elementary symmetric polynomials in the roots.

\begin{defn} \label{covering-def}
	The \textbf{covering number} of $f: X \tto Y$ is the smallest $k$ such that there exist open sets $\U_1, \cdots, \U_k \inc Y$ and continuous maps $g_i : \U_i \tto X$ such that $f \circ g = \Id_{\U_i}$
\end{defn}

\begin{rem} \label{cov-remark}
	$\pi$ is an $d!$-fold cover of $\Pd$ because $\pi(\gamma_1, \cdots, \gamma_d) = (\gamma_{\sigma(1)}, \cdots, \gamma_{\sigma(n)})$ where $\sigma \in \mathfrak{S}_d$, the symmetric group.
\end{rem}

We would like to only consider polynomials with distinct roots. Define
\[ \Delta := \{ (\gamma_1, \cdots, \gamma_d) : \exists \; i \not= j \text{ such that } \gamma_i = \gamma_j \} \]
Then, define $\Sigma := \pi(\Delta)$. $\Sigma$ consists of the polynomials with repeated roots, and is an algebraic variety (where the discriminant is zero).

Clearly $\pi$ restricts $\C^d \setminus \Delta \tto \Pd \setminus \Sigma$.

For the mechanics of the next theorem, we will need to change the usable inputs and outputs. We pick them very restrictively, and then inverse deformation retract onto $X$ and $Y$ as defined above. For now, the space of usable inputs will be defined as
\[ B_K := \{ f \in \Pd : |a_i| \leq K, i = 0, 1, \cdots, d-1, t^d + a_{d-1} t^{d-1} + \cdots + a_1 t + a_0 \} \]
with $K = K(d)$ sufficiently large so that if $f$ has all of its roots in the unit disk, then $f \in B_K$. Then, the set of acceptable outputs, $X$ is just the set of possible outputs from $B_K$ (the $d$-tuples which fall within $\epsilon$ of the roots of $f \in B_K$).

Now, we prove Smale's connection between computation theory and algebraic topology:
\begin{thm} [Smale] \label{thm-connection}
	The covering number of the restriction $\pi : \CmD \tto \PmS$ is less than or equal to the number of branches of Poly($d$) for all $\epsilon < \epsilon(d)$ described in the proof.
\end{thm}

\begin{proof}
Since any algorithm for Poly($d$) can be described as a computation tree, we can number its leaves $i = 1, \cdots, k$. Then, define $V_i \inc B_K$ as the set of polynomials which arrive at leaf $i$. So, since all polynomials from $B_K$ arrive at one of the leaves, and no polynomial arrives at multiple leaves, the following statements are true: $B_K = \cup_{i=1}^k V_i$ and $V_i \cap V_j = \emptyset$ if $i \not = j$.
	
Additionally, since the algorithm computes the roots for $f$, there is a very natural $\varphi_i : V_i \tto \C^d$ where $\varphi_i(f) = (z_1, \cdots, z_d)$ such that $|z_i - \gamma_i| < \epsilon$. $V_i$ is closed in $B_K$ (it is a semi-algebraic set). By the Tietze Extension Theorem, we can extend $\varphi_i : \U_i \tto B_K$ where $\varphi_i(f) = (z_1, \cdots, z_d)$ such that $|z_i - \gamma_i| < \epsilon$ still.

Now that we have a nice covering of $B_K$, we would like to ``blow this covering up'' so it covers $\PmS$. For this, we will need to inverse deformation retract from $B_K$ to the entire space.

\begin{defn} [Deformation retraction]
	A subspace $A \inc B$ is a deformation retraction of $B$ if there exists some homotopy $h : I \times B \tto B$ such that $h(0,\cdot) = \Id_B$ and $h(1,B) \inc A$ and $h(1, \cdot)$ restricted to $A$ is the identity on $A$. That is, $h(1,a) = a$ for all $a \in A$.
\end{defn}

We now need some Lemmas as technical tools:

\begin{lem}
	Let $A$ be a closed, compact subspace of a space $B$ such that ($B$,$A$) can be triangulated (it is homeomorphic to a simplicial complex and a subcomplex respectively). Then, there exists a neighborhood of $A$ such that $B \setminus N$ is a deformation retract of $A$.
\end{lem}

Then, let $S$ be the unit sphere in $\C^d$ under the standard Hermitian inner product.

\begin{lem}
	$(\pi(S), \Sigma \cap \pi(S))$ can be triangulated.
\end{lem}

\begin{lem}
	$\pi(S) \setminus \Sigma \cap \pi(S)$ is a deformation retract of $\PmS$.
\end{lem}

\begin{proof} [Proof of Lemma] We can just write down the deformation retraction. $h(t, x) = (1-t)x + t \frac{x}{\lVert x \rVert}$. (Note: technically $h$ acts on a vector in $\C^d$, but since it's invariant under the symmetric group, we can say that it acts on the polynomial space. That is, we can use $h$ to bring all of the roots of the polynomials into the unit sphere.)
\end{proof}

It follows that:
\begin{lem} \label{retract-poly-space}
	There exists a neighborhood, $N$, of $\Sigma \cap \pi(S)$ such that $\pi(S) \setminus N$ is a deformation retract of $\PmS$.
\end{lem}

As aforementioned, we're going to use this deformation retraction to expand our covering for $B_K$ to the covering for the entire space.

Let $h : \PmS \times I \tto \PmS$ be the deformation retraction from Lemma \ref{retract-poly-space}. That is, $h(\PmS, 1) \inc \pi(S) \setminus N$. We want to show that when we invert this retraction on the covering of $B_K$, we can uniquely determine a polynomial in $\PmS$ with the appropriate roots. So, choose $\mu(d)$ such that if $f \in \pi(S) \setminus N$, the roots of $f$ satisfy $|\gamma_i - \gamma_j| > \mu(d)$ for $i \not = j$.

Then, choose $P_i = \U_i \cap \pi(S) \setminus N$. Notice that the $P_i$ cover $\pi(S) \setminus N$.

Fix $\epsilon < \frac{\mu(d)}{2}$. This (slightly technical) work with roots and inequalities has been so we can uniquely determine the polynomial and nicely define the $g_i$ in Definition \ref{covering-def}. We can define $\psi_i(f)$ as follows (for $f \in P_i$):
\[ f \xrightarrow{\varphi_i(f)} (z_1, \cdots, z_d) \leadsto (\gamma_1, \cdots, \gamma_d) \]

Note that the determination of the roots above (denoted by $\leadsto$) is possible only because each $z_i$ has a closest root of $f$ defined unambiguously. Therefore, $\psi_i : P_i \tto \CmD$ is continuous and the $P_i$ cover $\pi(S) \setminus N$.

Finally, we define $Q_i := h^{-1} (P_i, 1)$ and using the covering homotopy property, we extend $\psi_i : Q_i \tto \CmD$.


\end{proof}

\begin{rem}
	Clearly, Theorem \ref{thm-connection} can be generalized for many other situations.
\end{rem}

\section{The cup length as a bound on the covering number}

We assume familiarity with the notions of cohomology and cohomology ring. We denote the cohomology ring of a topological space, $X$, as $H^\ast (X, R)$. We often omit $R$, and just write the cohomology ring as $H^\ast (X)$. In these cases, $R$ is irrelevant, and can just be taken to be $\mathbb{Z}$.

$H^\ast(X, R)$ comes equipped with a cup product. We write $\smile: H^\ast (X, R) \tto H^\ast(X, R)$ to be this map.

\begin{defn} [Cup length]
The \textbf{cup length} of a cohomology ring, $H^\ast (X; R)$ is the largest $k$ such that there exist $\eta_1, \cdots, \eta_k \in H^\ast (X; R)$ with $\eta_1 \smile \eta_2 \smile \cdots \smile \eta_k \not = 0$.
\end{defn}

In fact, we can bound the covering number of a map by the cup length of the kernel of the induced map on the cohomology.

\begin{lem} \label{cup-ln-ker}
	If $f : X \tto Y$ is continuous, then the covering number of $f$ is strictly greater than the cup length of $\ker(f^\ast)$ where $f^\ast : H^\ast(Y) \tto H^\ast(X)$ is the naturally induced map.	
\end{lem}

\begin{proof}
	Suppose by contradiction that there exists some covering, $V_1, \cdots, V_k$ of $Y$ and maps $\sigma_i : V_i \tto X$ with $f(\sigma_i(y)) = y$ and, the cup length is also $k$. So, there exist $\eta_1, \cdots, \eta_k \in \ker(f^\ast)$ such that $\eta_1 \smile \cdots \smile \eta_k \not = 0$.
	
	Then, the below sequence is exact, and the triangle commutes:
	\begin{center}
	\begin{tikzcd}
		\cdots \arrow[r] & H^\ast(Y, V_i) \arrow{r} & H^\ast(Y) \arrow[rd, "f^\ast"] \arrow{r} & H^\ast (V_i) \arrow[r] & \cdots \\
		& & & H^\ast(X) \arrow[u, "\sigma_i^\ast"]
	\end{tikzcd}
	\end{center}
	
	Then, we look at preimages of the $\eta_i$ in $H^\ast (Y, V_i)$. These exist since the sequence is exact; and, the image of $\eta_i$ in $H^\ast (V_i)$ is zero because the triangle commutes. We then choose $v_i \in H^\ast (Y, V_i)$ such that $v_i \mapsto \eta_i$ under the first map.
	
	Then, $v_1 \smile v_2 \smile \cdots \smile v_k \in H^\ast (Y, V_1 \cup V_2 \cup \cdots \cup V_k) = H^\ast (Y,Y) = 0$. But, this is a contradiction, since $v_1 \smile v_2 \smile \cdots \smile v_k \mapsto \eta_1 \smile \eta_2 \smile \cdots \smile \eta_k \not = 0$.
\end{proof}

\section{Computing the kernel of $\pi^\ast$}

We need another lemma to be able to apply Lemma \ref{cup-ln-ker} to our problem.

\begin{lem} \label{compute-ker-pi}
	Let $\pi : \CmD \tto \PmS$ as previously defined. Then, the induced map in cohomology (coefficients in $\ZZ$):
	\[ \pi^\ast : H^i(\PmS, \ZZ) \tto H^i (\CmD, \ZZ) \]
	is trivial for $i > 0$.
\end{lem}

\begin{proof}
	We need to use some groups, their classifying spaces, and other constructions from them. Let $O(n)$ be the orthogonal group, $BO(n)$ be the corresponding classifying space. Let $\mathfrak{S}_n$ be the symmetric group, and $B\mathfrak{S}_n = K(\mathfrak{S}_n, 1)$ be an Eilenberg-MacLane space. Finally, let $U_n \tto B\mathfrak{S}_n$ be the universal cover.
	
	It is a well-known result that $\PmS$ is an Eilenberg-MacLane space, $K(\pi_1(\PmS), 1)$. It is important to note that $\pi_1(\PmS) = \Bd(d)$, the braid group in $d$ strands. As per Remark \ref{cov-remark}, $\pi$ is a covering map with group $\mathfrak{S}_d$. Then, there's a natural map:
	\[ \pi_1(\PmS) = \Bd(d) \tto \mathfrak{S}_d \]
	This map is realized by taking the starting and ending positions of the strands. There's also a natural $\mathfrak{S}_d$ action on a basis, $e_1, \cdots, e_d \mapsto e_{\sigma(1)}, \cdots, e_{\sigma(d)}$ which realizes a map
	\[ \mathfrak{S}_d \tto O(d) \]
	Ring homomorphisms in cohomology over $\ZZ$ are induced by the group homomorphisms, so the above maps give rise to a map on the cohomologies:
	\[ H^\ast (\PmS) \leftarrow H^\ast (B\mathfrak{S}_d) \leftarrow H^\ast (BO(d)) \]
	According to Fuchs \cite{Fuchs}, the composition, $H^\ast (BO(d), \ZZ) \tto H^\ast (\PmS, \ZZ)$ is surjective.	Moreover, the composition
	\[ \pi_1(\CmD) \tto \pi_1(\PmS) \tto \pi_1(B\mathfrak{S}_d) \simeq \mathfrak{S}_d \]
	is trivial. So, by covering space theory, it can be lifted to the universal cover, as defined earlier. That is, there exists $\varphi$ such that
	\begin{center}
	\begin{tikzcd}
		\CmD \arrow[d, "\pi"] \arrow[r, "\varphi"] & U_d \arrow[d] \\
		\PmS \arrow[r] & B\mathfrak{S}_d
	\end{tikzcd}
	\end{center}
	commutes. Then, since all the cohomologies of $U_d$ are isomorphic to $H^0 (U_d, \ZZ)$, we know that the composition $H^i (B \mathfrak{S}_d, \ZZ) \tto H^i (\CmD, \ZZ)$ is trivial for $i > 0$. Then, the Lemma follows using the result from Fuchs.
\end{proof}

Define
\[ H^{>0} (\PmS, \ZZ) = \sum_{i=1}^{2d} H^i(\PmS, \ZZ) \]
Then, the last result we need is the core result of Fuchs.

\begin{lem} \label{compute-cup-ln}
	The cup length of $H^{>0} (\PmS, \ZZ)$ is greater than $(\log_2(d))^{2/3}$.
\end{lem}

\begin{proof}
	Fuchs \cite{Fuchs} proved that the generators of $H^{>0} (\PmS, \ZZ)$ are $g_{m,k}$ with $k= 0,1,2,\cdots$ and $m=1,2,3,\cdots$ with the degree of $g_{m,k}$ equal to $2^k \cdot 2^{m-1}$, and the following relations: 1) $g_{m,k}^2 = 0$ and otherwise, 2) $g_{m_1, k_1} \cdot g_{m_2, k_2} \cdots g_{m_n, k_n} = 0$ only if $2^{m_1 + m_2 + \cdots + m_n+ k_1 + k_2 + \cdots + k_n} > d$. 
	
	Taking the $\log_2$ of both sides shows that we are looking for pairs $(m_i, k_i)$ with $i=1, \cdots, n$ such that $n$ is as large as possible, and $\sum m_i + \sum k_i \leq \log_2 (d)$. To count these, we can first count the number of pairs $(m_i, k_i)$ that satisfy $m_i + k_i \leq N$ for some fixed $N$. A straightforward counting, shows that there are $\sum_1^N i = \frac{N (N+1)}{2}$ of these pairs.
	
	It is also straightforward that $\sum m_i + \sum k_i \leq \log_2 (d)$ will be satisfied provided
	\[ \sum_1^M i^2 = \frac{M(M+1)(2M+1)}{6} \leq \log_2(d) \]
	
	Finally, we check that $n = (\log_2(d))^{2/3}$ satisfies these conditions. Thus, the proof is complete.
\end{proof}

\section{Conclusion}

The proof of Theorem \ref{main-thm} now follows:

\begin{adjustwidth*}{0em}{2.5em}
	Topological complexity of Poly($d$) $+ 1$
	
	\begin{tabular}{p{7cm}|p{3cm}}
		$=$ Number of leaves in Poly($d$) & Definition \\
		$\geq$ Covering number of $\pi: \CmD \tto \PmS$ & Theorem \ref{thm-connection} \\
		$>$ Cup length of $\ker(\pi^\ast)$ & Lemma \ref{cup-ln-ker} \\
		$=$ Cup length of $H^{>0} (\PmS, \ZZ)$ & Lemma \ref{compute-ker-pi} \\
		$\geq (\log_2(d))^{2/3}$ & Lemma \ref{compute-cup-ln} 
	\end{tabular}
\end{adjustwidth*}

\section{Low-complexity algorithms via Newton's Method on the complex plane} \label{section-newtons-method}

Theorem \ref{main-thm} gives rise to an important problem. One wishes to find algorithms for computing the roots of polynomials with minimal topological complexity. Since a lower bound is given on the problem, it serves as the target complexity for an algorithm.

Moreover, finding such an algorithm gives an upper bound for the algorithm. We presently explore two methods of approximating roots. The first, discussed in this section, approximates roots of a polynomial via an extension of Newton's Method to the complex plane.


\subsection{Quadratic polynomials}

Given $f(t) = t^2 + a_1 t + a_0$, one can find its roots by applying the quadratic formula.
\[ t = \frac{-a_1 \pm \sqrt{\Omega}}{2} \]
In the above formula, $\Omega$ is the discriminant, $a_1^2 -4 a_0$. So, the problem essentially reduces to finding the square root of a complex number. This can be done if we restrict $\sqrt{\cdot} : \R \tto \R$ by applying Newton's method. We presently extend Newton's method to the complex numbers with a computation-theoretic approach.

Given a complex number $S$, we can approximate the solution to the equation $x^2 - S = 0$ by starting with a ``seed'', $x_0$, and then defining the sequence:
\[ x_{n+1} = x_n - \frac{x_n^2 - S}{2 x_n} = x_n - \frac{(x^2 - S)|_{x=x_n}}{\ddx (x^2 - S) |_{x=x_n}} \]
If the seed is chosen correctly, this sequence converges to $\sqrt{S}$. So, the only branching occurs in choosing the seed.

We created a program in Python to illustrate the duration of convergence when we start with $x_0 = 1$, and graph $S$ on the complex plane:

\begin{figure}[H]
\includegraphics[scale=0.5]{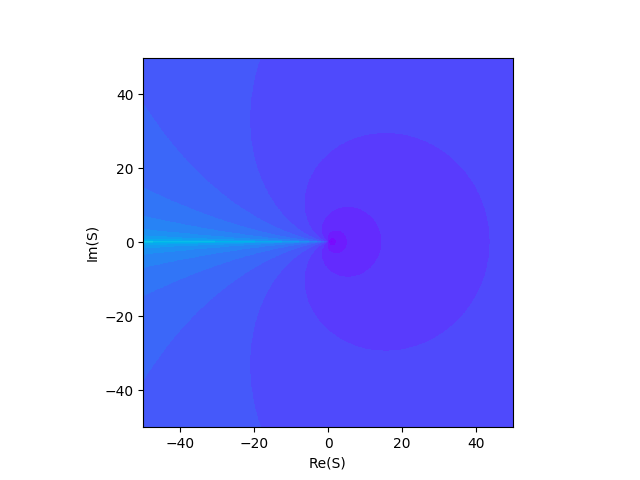}
\caption{The duration for convergence of $x_n \tto \sqrt{S}$ within a radius of $0.1$, seeded by the number $1$. Purple indicates rapid convergence, and light blue indicates divergence.}
\end{figure}

\begin{defn}
	A \textbf{fractal diagram for $t^d - S$ seeded by $k$ to threshold $r$} is a diagram (as above) showing the duration of convergence for $x_n \tto \sqrt[d]{S}$ within a radius of $r$, seeded by the number $k$ where purple indicates rapid convergence and light blue indicates divergence. The sequence $x_n$ is given by Newton's method extended to the complex plane as defined at the beginning of this section.
\end{defn}

There is clearly a problem with using this seed for $\Re(S) < 0$, especially on the negative $x$-axis. So, we can compute the same diagram for $x_0 = i$.

\begin{figure}[H]
\includegraphics[scale=0.5]{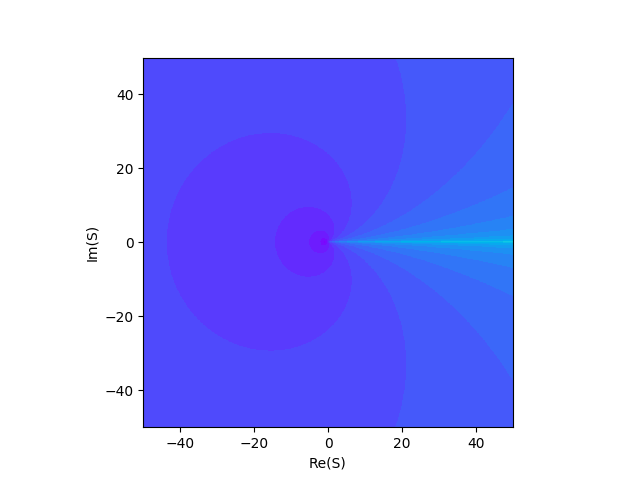}
\caption{Fractal diagram for $t^2 - S$ seeded by $i$ to threshold $0.1$.}
\end{figure}

We have therefore identified the branch in an efficient algorithm for Poly($d$): 

\begin{adjustwidth*}{2.5em}{2.5em}
$f(t) = t^2 + a_1 t + a_0$ \\
Define $\Omega = a_1^2 - 4 a_0$. \\
If $\Re(\Omega) < 0$: Perform Newton's Method to compute $\sqrt{\Omega}$ with seed $i$. \\
Otherwise: Perform Newton's Method to compute $\sqrt{\Omega}$ with seed $1$. \\
\textbf{Return} $t = \frac{1}{2}(-a_1 + \sqrt{\Omega}), \frac{1}{2}(-a_1 - \sqrt{\Omega})$.
\end{adjustwidth*}

This algorithm, has topological complexity $1$. Indeed, $(\log_2(2))^{2/3} -1 = 0$, so the number of branches is strictly greater than $0$, and this algorithm is therefore minimal.

\begin{rem}
	This algorithm does not violate Smale's no-loop condition because we can guarantee that for $S \in B_K$, Newton's method terminates in less than $N$ steps for some large $N$.
\end{rem}

\subsection{Cubic polynomials} \label{section-cubic-poly}
We similarly make use of the cubic formula, which requires computing the square root and the cube root. Specifically, the formula requires computing one square root, and two cube roots.

We can similarly compute the cube root of a number, $S$, by computing the sequence:
\[ x_{n+1} = x_n - \frac{x_n^3 - S}{3 x_n^2} \]
If $x_0$ is chosen properly, this sequence similarly converges to $\sqrt[3]{S}$. Using the same program, we created a similar graph as shown for the $\sqrt{\cdot}$ case.

\begin{figure}[H] \label{fig-3-seed-1}
\includegraphics[scale=0.5]{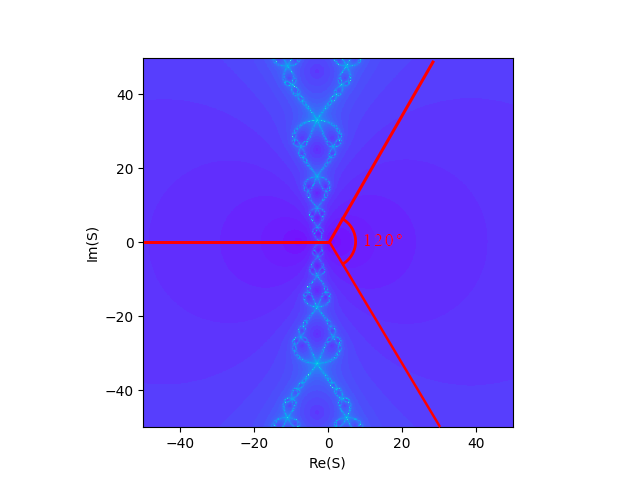}
\caption{Fractal diagram for $t^3 -S$ seeded by $1$ to threshold $0.1$.}
\end{figure}

As noted in the diagram, the algorithm, seeded by $1$, converges rapidly for $S$ in a specific third of the coordinate plane. We can modify the seed so that the algorithm converges in different sections.

\begin{figure}[H] \label{fig-3-seed-2pi-9}
\includegraphics[scale=0.5]{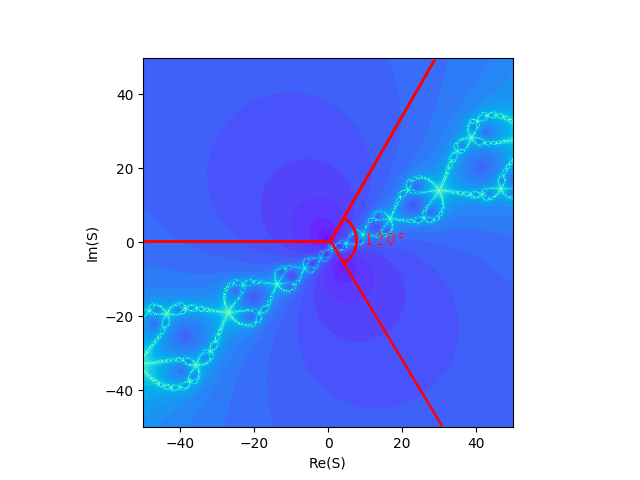}
\caption{Fractal diagram for $t^3 - S$ seeded by $e^{\frac{2 \pi}{9} i} \approx 0.77 + 0.64i$ to threshold $0.1$.}
\end{figure}

\begin{figure}[H] \label{fig-3-seed-4pi-9}
\includegraphics[scale=0.5]{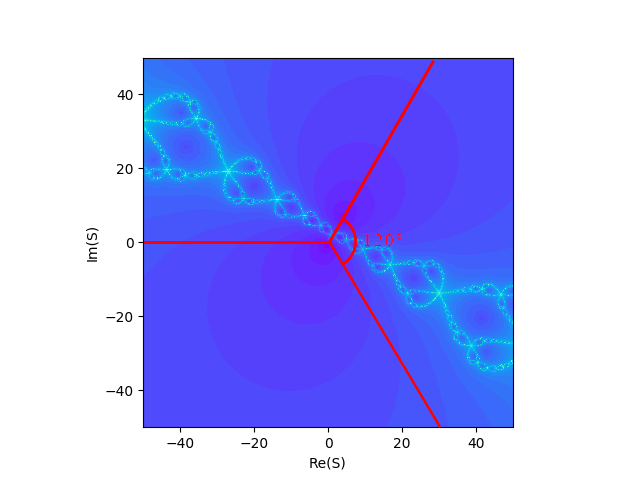}
\caption{Fractal diagram for $t^3 - S$ seeded by $e^{\frac{4 \pi}{9} i} \approx 0.17 + 0.98i$ to threshold $0.1$.}
\end{figure}

Therefore, we have cut the complex plane into three pieces, and identified a seed for each such that the algorithm converges rapidly in each when seeded appropriately. So, the corresponding algorithmic tree contains two branches.

Altogether, finding the roots of a general cubic equation requires at most five branches.

\subsection{Quartic polynomials} We can continue the approach of using closed form expressions for the roots of degree $d$ polynomials only up to $d=4$. In fact, for the quartic case, we must add additional branches, because there is a risk of dividing by zero, which does not occur in the other expressions.

Let $f(t) = t^4 + a_3 t^3 + a_2 t^2 + a_1 t + a_0$. We need to first perform several computations:
\begin{align*}
	p &= \frac{8a_2 - 3 a_3^2}{8} \\
	q &= \frac{a_3^3 - 4 a_3 a_2 + 8 a_1}{8} \\
	\Omega_0 &= a_2^3 - 3 a_3 a_1 + 12 a_0 \\
	\Omega_1 &= 2 a_2^3 - 9 a_3 a_2 a_1 + 27 a_3^2 a_0 + 27 a_1^2 - 72 a_2 a_0
\end{align*}
Thus far we have used no branches. Then define
\[ Q = \sqrt[3]{ \frac{\Omega_1 + \sqrt{\Omega_1^2 - 4 \Omega_0^3}}{2} } \]
using 3 branches. We would then like to define
\[ S = \frac{1}{2} \sqrt{-\frac{2}{3} p + \frac{1}{3} \left( Q + \frac{\Omega_0}{Q} \right)} \]
and ideally use just one more branch. But, we need to handle the case that $Q = 0$ with a separate branch. If $Q = 0$, though, at least three of the roots are equal, and the roots are rational functions in the coefficients. Thus far, then, we have used five branches.

If $Q \not = 0$, then
\begin{align}
	t_{1,2} &= -\frac{b}{4} - S \pm \sqrt{-4S^2 - 2p + \frac{q}{S}} \\
	t_{3,4} &= -\frac{b}{4} - S \pm \sqrt{-4S^2 - 2p - \frac{q}{S}}
\end{align}
using two more branches. Altogether, solving a quartic requires at least seven branches.

\subsection{$t^d - S$} The approach to solving general quadratic, cubic, and quartic equations was enabled entirely by computing square and cube roots. In general,

\begin{thm}
	The class of degree $d$, complex, monic polynomials which can be expressed as $t^d - S$ can be solved in $d$ branches.
\end{thm}

While a rigorous proof of this theorem is not given, we hope to provide significant insight into this result.

Firstly, the fractal diagram for $t^d - S$ seeded by $k$ to threshold $r$ contains $d-1$ ``branches'' for sufficiently small $r$. A ``branch'' on a fractal diagram refers to the visual branches evident in the diagram starting at a point near the origin and extending to infinity, growing in size. So, for a given seed, we can only guarantee convergence for polynomials in one of $d$ sections of the plane. This breakdown is done explicitly in Section \ref{section-cubic-poly}. Another example is given below.

\begin{figure}[H]
\includegraphics[scale=0.5]{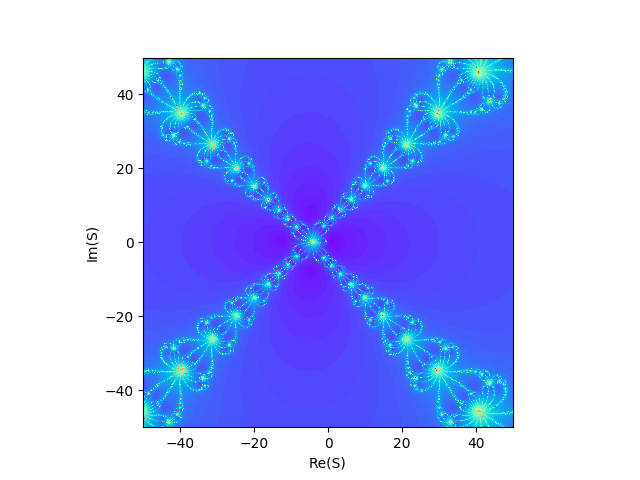}
\caption{Fractal diagram for $t^5 - S$ seeded by $1$ to threshold $0.1$.}
\end{figure}

Secondly, the fractal diagram for $t^d - S$ seeded by $e^{\frac{2 \pi}{d^2} i k}$ ($k = 0,1,2, \cdots, d-1$) is exactly the fractal diagram for $t^d - S$ seeded by $1$ rotated clockwise by $\frac{2 \pi k}{d}$ radians.

These two observations give a concrete algorithm to compute the roots of $t^d - S$ in $d$ branches with the branching based on the location of $S$ in the complex plane.

\begin{rem}
	Technically, the algorithm only gives one root of $t^d - S$, but the other roots are given by the roots of unity times the root the algorithm gives.
\end{rem}

\section{Low-complexity algorithms via Power Iteration} \label{section-power-iteration}
Given $f(t) \in B_K$ with $\deg(f) = d$ in coefficients $a_{d-1}, \cdots, a_0$, we can construct
\[ F = \begin{pmatrix}
	0 & 0 & \cdots & 0 & -a_0 \\
	1 & 0 & \cdots & 0 & -a_1 \\
	0 & 1 & \cdots & 0 & -a_2 \\
	\vdots & \vdots & \ddots & \vdots & \vdots \\
	0 & 0 & \cdots & 1 & -a_{d-1}
\end{pmatrix} \]
the companion matrix of $f(t)$. The characteristic polynomial of $F$ is $f$. Therefore, approximating the roots of the polynomial can be done by approximating the eigenvalues of $F$. One such algorithm to do this is called power iteration (see Bindel \cite{Bindel} for a more extensive definition).

The power iteration algorithm (almost always) returns an eigenvalue of a given ($n \times n$) matrix $A$. It proceeds as follows:

\circled{1} Choose a ``seed'' vector, $b_0$. In general, this vector is chosen randomly. For our specific case, with $F$ defined as above, the only vector we need to use is $(0, 0, \cdots, 0, 1)$. See Remark \ref{rem-last-basis-works} after the proof for an explanation as to why.

\circled{2} Compute the values of the sequence $b_n = \frac{A b_{n-1}}{\lVert A b_{n-1} \rVert}$.

``In general'', $b_n \tto v$ where $A v = \lambda v$. Formally:

\begin{thm}
	If $A$ is an $n \times n$ matrix, with eigenvalues $|\lambda_1| > |\lambda_2| > |\lambda_3| > \cdots > |\lambda_k| > 0$, then power iteration converges to a vector $v$ with $Av = \lambda_1 v$, unless $b_0$ is orthogonal to $v$.
\end{thm}

\begin{proof}
	We can write $A$ in Jordan canonical form, ordered such that:
	\[ J = \begin{pmatrix}
		J_1 & & &  \\
		& J_2 & & \\
		& & \ddots & \\
		& & & J_k
	\end{pmatrix} \]
	where $J_i$ is the Jordan block corresponding to $\lambda_i$. Written this way, $A = V J V^{-1}$. Then, we can write $b_0$ in the columns of $V$. Write $b_0 = \alpha_1 v_1 + \cdots + \alpha_n v_n$.
	
	Assume that $\alpha_1 \not = 0$, so $b_0$ has a nonzero component in the direction of $v_1$.
	
	Then, compute:
	\begin{align*}
		b_k &= \frac{A^k b_0}{\lVert A^k b_0 \rVert} \\
			&= \frac{(V J V^{-1})^k b_0}{\lVert (V J V^{-1})^k b_0 \rVert} \\
			&= \frac{V J^k V^{-1} (\alpha_1 v_1 + \cdots + \alpha_n v_n)}{\lVert V J^k V^{-1} (\alpha_1 v_1 + \cdots + \alpha_n v_n) \rVert} \\
			&= \left(\frac{\lambda_1}{| \lambda_1 |}\right)^k \frac{\alpha_1}{| \alpha_1 |} \frac{v_1 + \frac{1}{\alpha_1} V \left( \frac{1}{\lambda_1} J \right)^k (\alpha_2 e_2 + \cdots + \alpha_n e_n)}{\lVert v_1 + \frac{1}{\alpha_1} V \left( \frac{1}{\lambda_1} J \right)^k (\alpha_2 e_2 + \cdots + \alpha_n e_n) \rVert}
	\end{align*}
	And since $\lambda_1$ is dominant,
	\[ (\frac{1}{\lambda_1} J)^k \tto \begin{pmatrix}
		1 & 0 & \cdots & 0 \\
		0 & 0 & \cdots & 0 \\
		\vdots & \vdots & \ddots & \vdots \\
		0 & 0 & \cdots & 0
	\end{pmatrix} \]
	so, written as above, much of the fraction in our final expression for $b_k$ dies off as $k \tto \infty$. It follows that $b_k \tto \frac{v_1}{\lVert v_1 \rVert}$.
\end{proof}

\begin{rem} \label{rem-last-basis-works}
	If $|\lambda_1| > |\lambda_2| > |\lambda_3| > \cdots > |\lambda_k| > 0$, then power iteration always converges to the eigenvector corresponding to $\lambda_1$, unless $b_0$ is orthogonal to the eigenvector. Computationally, one chooses $b_0$ randomly to give a high likelihood that $b_0$ is not orthogonal to the eigenvector. 
	
	In a basis, at least one vector is not orthogonal to the eigenvector. So, if we perform power iteration on the entire canonical basis, at least one is guaranteed to converge. For $F$ above, it is clear that the vector $b_0 = (0, 0, \cdots, 0, 1)$ is sufficient to guarantee convergence, because all other basis vectors turn into that vector at some point during the algorithm ($F e_i = e_{i+1}$ for $i < n$).
\end{rem}

\begin{rem} \label{rem-geom-conv}
	Power iteration converges geometrically, with ratio $| \frac{\lambda_2}{\lambda_1} |$.
\end{rem}

Notice that Remark \ref{rem-geom-conv} is especially problematic. It means that we cannot guarantee that power iteration converges quickly, which is important, since Poly($d$) excludes programs with loops in them. 

Earlier in the paper (e.g., Section \ref{section-newtons-method}), we have made the notion of ``quick'' convergence (i.e. no loops in the program) a mathematical statement. We would like to guarantee that there is a universal $N$ such that after $N$ computations, the algorithm will have arrived within $\epsilon$ of the correct answer. For power iteration, though, we must make a different statement.

We can conclude from Remark \ref{rem-geom-conv} that there exists some sufficiently large $N = N(\delta)$ such that for all $1 > \delta >0$, we can construct a program, based on power iteration which can find all the roots of $f$ with \textit{no} branches, many iterations, and $b_0 = (0, 0, \cdots, 0, 1)$, terminating in $N$ steps, which finds the roots to within $\epsilon$ with \textit{probability} greater than $\delta$.

\begin{rem}
	The challenge is only to find a single root of $f$. Then, we can divide $f$ by $t - \lambda$, and repeat the procedure.
\end{rem}

The careful reader will notice that above, we exclude the case where $|\lambda_1| = |\lambda_2|$. This case is problematic.

\subsection{Equal-magnitude dominant and subdominant eigenvalues}

If $|\lambda_1| = |\lambda_2|$, but $\lambda_1 \not= \lambda_2$ (i.e. $\lambda_1 = r e^{\theta_1 i}$ and $\lambda_2 = r e^{\theta_2 i}$ with $\theta_1 \not = \theta_2$), then power iteration does not converge to a single vector. Instead, it eventually approaches the sequence
\[ e^{n \theta_1 i} + e^{n \theta_2 i}\]
In this case the goal is to isolate $\theta_1$ and $\theta_2$ from the values of the sequence. We suspect this is difficult, and cannot be done without using many branches.

\subsection*{Acknowledgments}  I would like to thank my mentor, Weinan Lin, for teaching me much of the material presented here, and Shmuel Weinberger for giving me this topic and directing me appropriately. I would like to especially thank J. Peter May for organizing the REU at the University of Chicago, which I attended during the summer of 2017.

\end{document}